\documentclass[11pt]{article}
\usepackage[margin=1in]{geometry}
\usepackage[T1]{fontenc}
\usepackage[utf8]{inputenc}
\usepackage{lmodern,amsmath,amssymb,amsthm,booktabs,microtype}
\usepackage{tikz,tabularx}
\usepackage[numbers,sort&compress]{natbib}
\usepackage[colorlinks=true,linkcolor=blue,citecolor=blue,urlcolor=blue]{hyperref}
\newtheorem{theorem}{Theorem}[section]
\newtheorem{lemma}[theorem]{Lemma}
\newtheorem{proposition}[theorem]{Proposition}
\theoremstyle{remark}
\newtheorem{remark}[theorem]{Remark}
\newtheorem{question}[theorem]{Question}
\newcommand{\RP}{\mathbb{RP}^{2}}
\newcommand{\Ftwo}{\mathbb F_{2}}
\newcommand{\ew}{\operatorname{ew}}
\newcommand{\fw}{\operatorname{fw}}

\newcommand{\diam}{\operatorname{diam}}
\title{Power-of-two cycles in projective cubic graphs\protect\\ of large edge-width}
\author{Jean-Fran\c{c}ois Culus\\[4pt]
\small MEMIAD, Universit\'e des Antilles, Schoelcher, France\\
\small CREC Saint-Cyr, Guer, France}
\date{}
\begin{document}
\maketitle

\begin{abstract}
Heckman and Krakovski proved that every $3$-connected cubic planar graph
contains a cycle of length $2^m$ for some $2\le m\le7$. We extend this to
$3$-connected cubic graphs embedded in the real projective plane with
face-width at least three and edge-width at least $48$; the cycle may
moreover be taken contractible. The argument passes to the spherical
double cover, where the planar theorem provides the cycle inside a
subgraph of diameter at most $47$, and the condition on the edge-width
guarantees that this subgraph projects injectively.
\end{abstract}

\section{Introduction}

In 1995 Erd\H{o}s and Gy\'arf\'as asked whether every graph of minimum
degree at least three contains a cycle whose length is a power of
two~\cite{Erdos1997,HK}. The question remains open, although it has been
answered for a number of classes of graphs. Shauger~\cite{Shauger1998}
settled it for $K_{1,m}$-free graphs of minimum degree at least $m+1$ or
maximum degree at least $2m-1$, and Daniel and
Shauger~\cite{DanielShauger2001} for planar claw-free graphs of minimum
degree at least three. Markstr\"om~\cite{Markstrom2004} verified it for
all cubic graphs on at most $29$ vertices, and found that the smallest
cubic planar graph with no cycle of length $4$ or $8$ has $24$ vertices.
At the other end of the spectrum, Liu and
Montgomery~\cite{liu2023solutionerdoshajnal} proved that graphs of
sufficiently large average degree contain a cycle whose length is a power
of two, so that the conjecture is really a question about sparse graphs.
In this direction, Carr~\cite{carr2026predominantlycubic} showed that at
least $4/7$ of the vertices of a minimal counterexample have degree three,
and Ducoffe and
Dumitru~\cite{ducoffe2026structuredsearcherdhosgyarfascounterexamples}
raised this proportion above $2/3$; they also verified the conjecture for
cubic graphs on at most $48$ vertices.

For cubic graphs, the planar case was settled by Heckman and
Krakovski~\cite{HK}, who proved that every $3$-connected cubic planar
graph has a cycle of length $2^m$ with $2\le m\le7$. Their proof is a
long discharging argument, parts of which were carried out by computer.
The range of exponents cannot be reduced to $2\le m\le4$: Exoo
constructed a $3$-connected cubic planar graph on $420$ vertices with no
cycle of length $4$, $8$ or $16$~\cite{exoo2014graphserdhosgyarfasconjecture}.

In their concluding remarks, Heckman and Krakovski consider replacing
planarity by projective-planarity. Discharging remains available, since
the projective plane has positive Euler characteristic, but they expect
many of their local arguments to become more delicate~\cite[Section~6]{HK}.
The purpose of this note is to show that, when the embedding is wide
enough, no new local analysis is needed: the planar theorem can be
transferred as it stands.

\paragraph{The result.}
We work with the usual model of the real projective plane: a closed disk
in which antipodal points of the boundary are identified
(Figure~\ref{fig:k6-disk}). A chord joining two antipodal points closes
up into a curve whose neighbourhood is a M\"obius band, and up to
homeomorphism every essential, that is, non-contractible, simple closed
curve is of this kind. For a graph $G$ embedded in the projective plane,
the \emph{face-width} $\fw(G)$ is the least number of points in which an
essential simple closed curve, meeting $G$ in finitely many points,
intersects $G$; passing through vertices is allowed. The
\emph{edge-width} $\ew(G)$ is the length of a shortest essential cycle of
$G$. A minimum over an empty family is $+\infty$. We refer to Mohar and
Thomassen~\cite{MoharThomassen2001} for background. All graphs are finite
and simple, and lengths and distances are measured in edges.

\begin{theorem}\label{thm:ew48}
Let $G$ be a cubic $3$-connected graph embedded in the real projective
plane with $\fw(G)\ge3$. If $\ew(G)\ge48$, then $G$ contains a
contractible cycle of length $2^m$ for some $2\le m\le7$.
\end{theorem}

The proof rests on the local form of the planar theorem. The orientable
double cover of the projective plane is the sphere, and $G$ lifts to a
cubic plane graph $\widetilde G$ to which the theorem of Heckman and
Krakovski applies. What they prove is slightly stronger than the
existence of a cycle: the cycle can be found in a subgraph $H$ all of
whose vertices lie within distance six of a face of length at most $71$,
and such a subgraph has diameter at most $6+35+6=47$. Two vertices of $H$
lying over the same vertex of $G$ would be joined in $H$ by a path of
length at most $47$; this path would project to an essential closed walk,
and hence to an essential cycle shorter than the edge-width. Thus $H$
projects injectively, and the image of its power-of-two cycle is a cycle
of the same length. Most of the work lies in checking that $\widetilde G$
satisfies the hypotheses of the planar theorem, and this is where the
criterion of Mohar and Vodopivec is used. The case of small edge-width
is discussed in Section~\ref{sec:remarks}.

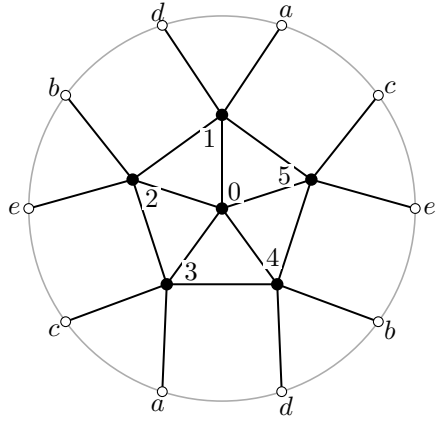
\begin{figure}[tbp]
\centering
\begin{tikzpicture}[x=1cm,y=1cm,line cap=round,line join=round,
  graph edge/.style={black,line width=.75pt},
  vertex/.style={circle,fill=black,inner sep=1.65pt},
  boundary point/.style={circle,draw=black,fill=white,inner sep=1.3pt},
  vertex label/.style={font=\small,fill=white,inner sep=1pt},
  boundary label/.style={font=\small}]
  \def\diskradius{2.55}
  \def\ringradius{1.24}
  \draw[gray!65,line width=.6pt] (0,0) circle (\diskradius);
  \coordinate (c) at (0,0);
  \foreach \i in {0,...,4} {
    \coordinate (v\i) at ({90+72*\i}:\ringradius);
    \coordinate (p\i) at ({72+72*\i}:\diskradius);
    \coordinate (q\i) at ({252+72*\i}:\diskradius);
  }
  \foreach \i in {0,...,4} {
    \pgfmathtruncatemacro{\j}{mod(\i+1,5)}
    \pgfmathtruncatemacro{\k}{mod(\i+2,5)}
    \draw[graph edge] (c)--(v\i)--(v\j);
    \draw[graph edge] (v\i)--(p\i);
    \draw[graph edge] (q\i)--(v\k);
  }
  \node[vertex] at (c) {};
  \node[vertex label] at (54:.28) {$0$};
  \foreach \i in {0,...,4} {
    \pgfmathtruncatemacro{\lab}{\i+1}
    \node[vertex] at (v\i) {};
    \path (v\i) ++({243+72*\i}:.36) node[vertex label] {$\lab$};
  }
  \foreach \i/\letter in {0/a,1/b,2/c,3/d,4/e} {
    \node[boundary point] at (p\i) {};
    \node[boundary point] at (q\i) {};
    \node[boundary label] at ({72+72*\i}:{\diskradius+.19}) {$\letter$};
    \node[boundary label] at ({252+72*\i}:{\diskradius+.19}) {$\letter$};
  }
\end{tikzpicture}
\caption{The complete graph $K_6$ in the disk model of the projective
plane. Boundary points with the same letter are identified; they are not
vertices. The six vertices, fifteen edges and ten triangular faces give Euler
characteristic $6-15+10=1$. The dual of this map is the Petersen graph
embedded in the projective plane.}
\label{fig:k6-disk}
\end{figure}

\section{Embeddings and the spherical lift}\label{sec:lift}

An embedding is \emph{cellular} if every face is an open disk, and
\emph{polyhedral} if moreover every facial walk is a cycle and the
closures of any two distinct faces are disjoint or meet in a single
vertex or a single edge. We use the same letter for a face and its
boundary cycle, and $|f|$ is its length.
For cellular embeddings, Mohar and Vodopivec showed that a cubic graph is
polyhedrally embedded exactly when it is $3$-connected and the face-width
is at least three~\cite[Proposition~1.1]{MV}. To apply this criterion we need cellularity; in the projective plane it
is already implied by positive face-width.

\begin{lemma}\label{lem:cellular}
Every embedding of a connected graph in $\RP$ with positive face-width is
cellular.
\end{lemma}

The proof, a short Euler characteristic count, is in
Appendix~\ref{app:topology}. Let $p\colon S^2\to\RP$ be the double cover,
$\tau$ its deck involution, and $\widetilde G=p^{-1}(G)$.

\begin{proposition}\label{prop:lift}
If $G$ is cubic and polyhedrally embedded in $\RP$, then $\widetilde G$
is a connected cubic graph, polyhedrally embedded in $S^2$. In particular
$\widetilde G$ is $3$-connected and planar.
\end{proposition}

Again the proof is in Appendix~\ref{app:topology}. Note that we obtain
the $3$-connectivity of $\widetilde G$ from its embedding, through the
criterion of Mohar and Vodopivec, rather than from general facts about
coverings of graphs.

\section{Proof of Theorem~\ref{thm:ew48}}\label{sec:proof}

We use the following local form of the planar theorem.

\begin{theorem}[Heckman and Krakovski {\cite[Result~1.2]{HK}}]\label{thm:HK}
Every cubic $3$-connected plane graph has a face $f$ with $|f|\le71$ and a
subgraph $H\supseteq f$ such that every vertex of $H$ is joined to $f$ by a
path in $H$ of length at most six, and $H$ contains a cycle of length
$2^m$ for some $2\le m\le7$.
\end{theorem}

Heckman and Krakovski measure the length of a path by its number of
vertices~\cite[p.~2]{HK}. With this convention their paths have at most
five edges, the bound of Lemma~\ref{lem:diameter} becomes $45$, and
Theorem~\ref{thm:ew48} holds as soon as $\ew(G)\ge46$. We use the bound of
six edges, which is valid under either reading.

\begin{lemma}\label{lem:diameter}
For $f$ and $H$ as in Theorem~\ref{thm:HK}, $\diam(H)\le 47$.
\end{lemma}

\begin{proof}
Join $x,y\in V(H)$ to vertices $x',y'$ of $f$ by paths in $H$ of length at
most six, and $x'$ to $y'$ by the shorter arc of $f$, of length at most
$\lfloor 71/2\rfloor=35$. This gives a walk in $H$ of length at most $47$
from $x$ to $y$.
\end{proof}

We shall use that the fundamental group of the projective plane has
order two, and that the natural map $\pi_1(\RP)\to H_1(\RP;\Ftwo)\cong\Ftwo$
is an isomorphism. In particular, a closed walk in $G$ lifts to a closed
walk in $\widetilde G$ if and only if it is contractible; otherwise its
lift joins the two preimages of its initial vertex.

\begin{lemma}\label{lem:walk}
An essential closed walk in a graph embedded in $\RP$ contains an
essential cycle no longer than the walk.
\end{lemma}

\begin{proof}
Repeatedly cut the walk at repeated vertices. The pieces are cycles and edges
traversed back and forth; their lengths add up to the length of the walk,
and their classes in $H_1(\RP;\Ftwo)$ add up to the class of the walk,
which is nonzero. The back-and-forth pieces are trivial, so one of the
cycles is essential.
\end{proof}

\begin{lemma}\label{lem:injective}
If $T$ is a connected subgraph of $\widetilde G$ with $\diam(T)<\ew(G)$,
then $p$ is injective on $V(T)$.
\end{lemma}

\begin{proof}
Suppose $x\ne y$ in $V(T)$ have the same image, so that $y=\tau(x)$. A
shortest path from $x$ to $y$ in $T$ projects to a closed walk in $G$ of
length at most $\diam(T)$ whose lift is not closed. This walk is
essential, and Lemma~\ref{lem:walk} yields an essential cycle shorter
than $\ew(G)$, which is absurd.
\end{proof}

\begin{proof}[Proof of Theorem~\ref{thm:ew48}]
By Lemma~\ref{lem:cellular} the embedding is cellular, so it is
polyhedral by the criterion of Mohar and Vodopivec, and by
Proposition~\ref{prop:lift} the lift $\widetilde G$ is a cubic
$3$-connected plane graph. Theorem~\ref{thm:HK} gives $f$, $H$ and a cycle
$D\subseteq H$ of length $2^m$ with $2\le m\le7$. Since
$\diam(H)\le47<\ew(G)$, Lemma~\ref{lem:injective} shows that $p$ is
injective on $V(H)$, so $p(D)$ is a cycle of length $2^m$ in $G$. It is
contractible, since $D$ is contractible in $S^2$ and $p$ maps a
null-homotopy of $D$ to one of $p(D)$.
\end{proof}

The argument only uses the two constants of Theorem~\ref{thm:HK}: a face
of length at most $F$ and paths of length at most $\rho$ give the
threshold $2\rho+\lfloor F/2\rfloor+1$.

\section{Remarks and the case of small edge-width}\label{sec:remarks}

\paragraph{Two cautionary examples.}
An arbitrary cycle of the lift need not project to a cycle. The Petersen
graph embeds in the projective plane with the dodecahedron as its lift,
and an essential pentagon lifts to a $10$-cycle that winds twice around
it. Separation arguments call for similar care. In the hemi-cube, the
quotient of the cube by the antipodal map, the graph is $K_4$ with faces
$ABCD$, $ABDC$ and $ACBD$; the cycle $ABCD$ bounds a disk, and yet the
disjoint edges $AC$ and $BD$ join alternate vertices of it, both running
through the M\"obius-band side. (This embedding is not polyhedral, and
$K_4$ contains a $4$-cycle in any case.) Neither phenomenon arises in our
argument: the planar theorem is applied in the lift, and
Lemma~\ref{lem:injective} ensures that the local witness projects
injectively.

\paragraph{Small edge-width.}
Below the threshold, we do not know the answer to the following question.

\begin{question}\label{q:small-width}
Does every cubic $3$-connected graph embedded in $\RP$ with face-width at
least three contain a cycle of length $2^m$ for some $2\le m\le7$?
\end{question}

By Theorem~\ref{thm:ew48}, a counterexample has a shortest essential
cycle $C_0$ of length at most $47$, and every planar witness in the lift
meets $p^{-1}(C_0)$. Write $B_r(X)$ for the set of vertices of $G$ at
distance at most $r$ from a vertex set $X$.

\begin{remark}\label{rem:localization}
Let $G$ be a counterexample to Question~\ref{q:small-width}, and let
$f$, $H$ be given by Theorem~\ref{thm:HK} in the lift. Then $H$ meets
$p^{-1}(C_0)$, every vertex of $p(H)$ lies in $B_{47}(C_0)$, every vertex
of $p(f)$ lies in $B_{41}(C_0)$, and some vertex of $p(f)$ is at distance
at most six from $C_0$. Indeed, the complement of $C_0$ is an open disk
whose preimage consists of two disjoint disks, each mapped
homeomorphically onto it; if $H$ avoided $p^{-1}(C_0)$, it would lie in
one of them and project injectively, together with its power-of-two
cycle. Choose a vertex $z$ of $H$ in $p^{-1}(C_0)$. Every vertex of $H$ is
within distance $47$ of $z$, and a path of length at most six joins $z$
to $f$; adding an arc of $f$ of length at most $35$ gives the remaining
bounds.
\end{remark}

This localizes the planar witnesses, but gives no bound on the order of
a counterexample.

\section*{Use of AI tools}

This research was carried out with the help of AI tools. The research
sessions were run with the MerLean
harness~\cite{ren2026merleanagenticframeworkautoformalization,li2026merleanproverrecursiveloopingharness}
in OpenAI Codex, calling external models from DeepSeek and Google;
Claude (Anthropic) was used to adapt the harness and to check results.
The extension lemma and the
sumset inclusion for the twelve profiles listed in
Appendix~\ref{app:arithmetic}
were formalized in Lean~4; the main theorem has not been formally
verified. The author takes full responsibility for the content of this
note.

\appendix
\section{A sumset lemma}\label{app:arithmetic}

The proof of \cite[Lemma~4.3]{HK} uses a sumset lemma \cite[Lemma~4.2]{HK}
for the sets of length increments attached to clusters in
\cite[Lemma~4.1]{HK}. For completeness we record, with a self-contained
proof, the form in which it is applied to these sets.

Intervals are integer intervals, $[a,b]=\{a,\ldots,b\}$, and
$A+B=\{a+b:a\in A,\ b\in B\}$. A set $A\subseteq[0,M]$ has \emph{isolated
holes} if no two consecutive integers of $[0,M]$ are missing from $A$.

\begin{lemma}\label{lem:extension}
Let $B\subseteq[0,M]$ contain $0$ and $M$ and have isolated holes. If
$[0,K]\subseteq A$ with $K\ge1$, then $[0,K+M]\subseteq A+B$.
\end{lemma}

\begin{proof}
Let $0\le d\le K+M$. If $d\le M$, then $d\in B$ or $d-1\in B$, and
$d=0+d$ or $d=1+(d-1)$. If $d>M$, then $d=(d-M)+M$ with $1\le d-M\le K$.
\end{proof}

\begin{proposition}\label{prop:sumset}
Let $t\ge2$ and let $A_i\subseteq[0,M_i]$, $1\le i\le t$, have isolated
holes and contain $0,1,M_i-1,M_i$, where $M_i\ge1$. Then
$A_1+\cdots+A_t=[0,M_1+\cdots+M_t]$.
\end{proposition}

\begin{proof}
After interchanging $A_1$ and $A_2$ if necessary, write $A=A_1$,
$B=A_2$ and $M=M_1\le N=M_2$. Since $B$ contains $0$ and
$N$ and has isolated holes, $B+\{0,1\}=[0,N+1]$. As $A$ contains $0,1$
and $M-1,M$, the sum $A+B$ contains $[0,N+1]$ and $[M-1,M+N]$, which
overlap; so $A+B=[0,M+N]$. Adding the other sets one by one with
Lemma~\ref{lem:extension} gives one inclusion, and the maxima give the
other.
\end{proof}

The twelve increment sets of \cite[Lemma~4.1]{HK} for the cluster types
in the family $\mathfrak I_F$ are
\[
\begin{gathered}
{}[0,1],\ [0,3],\ [0,4],\ [0,4]\setminus\{2\},\ [0,5]\setminus\{2\},\
[0,6],\ [0,6]\setminus\{2\},\\
[0,7]\setminus\{2,5\},\ [0,8]\setminus\{5\},\ [0,9]\setminus\{2\},\
[0,10],\ [0,14],
\end{gathered}
\]
and each satisfies the hypotheses of Proposition~\ref{prop:sumset}. Hence
any sum of at least two of them, repetitions allowed, is a full interval;
this is the statement needed in the proof of \cite[Lemma~4.3]{HK}.

\section{Topological details}\label{app:topology}

\begin{proof}[Proof of Lemma~\ref{lem:cellular}]
Let $N$ be a connected closed regular neighbourhood of the graph, with $b$ boundary
curves, and let $M_1,\ldots,M_r$ be the closures of the components of
$\RP\setminus N$, so that $b\ge r$. Write $e(S)$ for the Euler genus of a
compact connected surface $S$ (twice the genus if $S$ is orientable, the number of
crosscaps otherwise); then $\chi(S)=2-e(S)$ minus the number of boundary
curves of $S$. Since $\chi(\RP)=1$ and $\chi$ is additive along circles,
\[
 e(N)+\sum_{i=1}^r e(M_i)+2(b-r)=1 .
\]
Hence $b=r$, each $M_i$ has one boundary curve, and exactly one of the
surfaces has Euler genus one, all others zero. If it were some $M_i$, it
would be a M\"obius band whose core is an essential curve missing the
graph, contradicting positive face-width. So every $M_i$ is a disk, which
is cellularity.
\end{proof}

\begin{proof}[Proof of Proposition~\ref{prop:lift}]
The closure of each face of $G$ is a closed disk, so its preimage is two
disjoint closed disks mapped homeomorphically onto it; these are the
closures of the faces of $\widetilde G$, and their boundaries are cycles.
Two lifts of the same face are disjoint. Two lifts of different faces meet
inside the preimage of the intersection of their closures, and by uniqueness of lifting their
intersection is empty, a vertex, or a whole edge. So the lifted
embedding is polyhedral.

The lift is cubic since $p$ is a local homeomorphism. It has no loops,
and two edges with the same ends would project onto the same edge of $G$
and, by uniqueness of lifting from an endpoint, coincide; so
$\widetilde G$ is simple. It is connected, because the faces are disks
attached along connected boundaries and $S^2$ is connected. The criterion
of Mohar and Vodopivec now gives $3$-connectivity, and removing a point
inside a face gives a plane embedding.
\end{proof}


\begin{thebibliography}{13}
\providecommand{\natexlab}[1]{#1}
\providecommand{\url}[1]{\texttt{#1}}
\expandafter\ifx\csname urlstyle\endcsname\relax
  \providecommand{\doi}[1]{doi: #1}\else
  \providecommand{\doi}{doi: \begingroup \urlstyle{rm}\Url}\fi

\bibitem[Carr(2026)]{carr2026predominantlycubic}
Avery Carr.
\newblock Every Minimal Counterexample to the {Erd{\H o}s--Gy{\'a}rf{\'a}s}
  Conjecture is Predominantly Cubic, 2026.
\newblock URL \url{https://arxiv.org/abs/2605.22844}.

\bibitem[Daniel and Shauger(2001)]{DanielShauger2001}
Dale Daniel and Stephen~E. Shauger.
\newblock A result on the {Erd{\H o}s--Gy{\'a}rf{\'a}s} conjecture in planar
  graphs.
\newblock \emph{Congressus Numerantium}, 153:\penalty0 129--139, 2001.

\bibitem[Ducoffe and
  Dumitru(2026)]{ducoffe2026structuredsearcherdhosgyarfascounterexamples}
Guillaume Ducoffe and Bogdan Dumitru.
\newblock Towards a more structured search for Erd\H{o}s--Gy\'arf\'as
  counter-examples, 2026.
\newblock URL \url{https://arxiv.org/abs/2609.28594}.

\bibitem[Erd{\H o}s(1997)]{Erdos1997}
Paul Erd{\H o}s.
\newblock Some old and new problems in various branches of combinatorics.
\newblock \emph{Discrete Mathematics}, 165--166:\penalty0 227--231, 1997.
\newblock \doi{10.1016/S0012-365X(96)00173-2}.
\newblock URL \url{https://doi.org/10.1016/S0012-365X(96)00173-2}.

\bibitem[Exoo(2014)]{exoo2014graphserdhosgyarfasconjecture}
Geoffrey Exoo.
\newblock Three Graphs and the Erd\H{o}s--Gy\'{a}rf\'{a}s Conjecture, 2014.
\newblock URL \url{https://arxiv.org/abs/1403.5636}.

\bibitem[Heckman and Krakovski(2013)]{HK}
Christopher~Carl Heckman and Roi Krakovski.
\newblock {Erd{\H o}s--Gy{\'a}rf{\'a}s Conjecture for Cubic Planar Graphs}.
\newblock \emph{The Electronic Journal of Combinatorics}, 20\penalty0
  (2):\penalty0 P7, 2013.
\newblock \doi{10.37236/3252}.

\bibitem[Li et~al.(2026)Li, Zhu, and
  Ren]{li2026merleanproverrecursiveloopingharness}
Jinzheng Li, Zeru Zhu, and Yuanjie Ren.
\newblock MerLean-Prover: A Recursive Looping Harness for Lean 4 Theorem
  Proving, 2026.
\newblock URL \url{https://arxiv.org/abs/2605.26959}.

\bibitem[Liu and Montgomery(2023)]{liu2023solutionerdoshajnal}
Hong Liu and Richard Montgomery.
\newblock A solution to {Erd{\H o}s} and {Hajnal}'s odd cycle problem.
\newblock \emph{Journal of the American Mathematical Society}, 36:\penalty0
  1191--1234, 2023.
\newblock \doi{10.1090/jams/1018}.

\bibitem[Markstr{\"o}m(2004)]{Markstrom2004}
Klas Markstr{\"o}m.
\newblock Extremal graphs for some problems on cycles in graphs.
\newblock \emph{Congressus Numerantium}, 171:\penalty0 179--192, 2004.

\bibitem[Mohar and Thomassen(2001)]{MoharThomassen2001}
Bojan Mohar and Carsten Thomassen.
\newblock \emph{Graphs on Surfaces}.
\newblock Johns Hopkins Studies in the Mathematical Sciences. The Johns Hopkins
  University Press, Baltimore and London, 2001.
\newblock ISBN 9780801866890.
\newblock \doi{10.56021/9780801866890}.
\newblock URL \url{https://www.press.jhu.edu/books/title/1675/graphs-surfaces}.

\bibitem[Mohar and Vodopivec(2006)]{MV}
Bojan Mohar and Andrej Vodopivec.
\newblock {On Polyhedral Embeddings of Cubic Graphs}.
\newblock \emph{Combinatorics, Probability and Computing}, 15\penalty0
  (6):\penalty0 877--893, 2006.
\newblock \doi{10.1017/S0963548306007607}.

\bibitem[Ren et~al.(2026)Ren, Li, and
  Qi]{ren2026merleanagenticframeworkautoformalization}
Yuanjie Ren, Jinzheng Li, and Yidi Qi.
\newblock MerLean: An Agentic Framework for Autoformalization in Quantum
  Computation, 2026.
\newblock URL \url{https://arxiv.org/abs/2602.16554}.

\bibitem[Shauger(1998)]{Shauger1998}
Stephen~E. Shauger.
\newblock Results on the {Erd{\H o}s--Gy{\'a}rf{\'a}s} conjecture in
  {$K_{1,m}$}-free graphs.
\newblock \emph{Congressus Numerantium}, 134:\penalty0 61--65, 1998.

\end{thebibliography}
\end{document}